\documentclass[a4paper,10pt,twoside]{article}

\usepackage{amssymb}
\usepackage{amsthm}
\usepackage{amsmath}
\usepackage{enumerate}
\usepackage{hyperref}
\usepackage{graphicx}

\newtheorem{thmA}{Theorem}
\newtheorem{thm}{Theorem}
\newtheorem{lemA}{Lemma}
\newtheorem{lem}{Lemma}

\begin{document}

\title{An approximation of It\^o diffusions based on simple random walks}

\author{John van der Hoek\footnote{Address: School of Mathematics and Statistics, University of South Australia, City West Campus, Yungodi Bldg, Level 3, Room 73, GPO Box 2471, Adelaide, South Australia 5001, Australia, e-mail: john.vanderhoek@unisa.edu.au} \\
University of South Australia \\
and Tam\'as Szabados\footnote{Address:
Department of Mathematics, Budapest University of Technology and Economics, M\H{u}egyetem rkp. 3, H
\'ep. V em. Budapest, 1521, Hungary, e-mail: szabados@math.bme.hu, telephone: (+36 1)
463-1111/ext. 5907, fax: (+36 1) 463-1677} \\
Budapest University of Technology and Economics}


\maketitle

\bigskip


\begin{abstract}
The aim of this paper is to develop a sequence of discrete approximations to a one-dimensional It\^o diffusion that almost surely converges to a weak solution of the given stochastic differential equation. Under suitable conditions, the solution of the stochastic differential equation can be reduced to the solution of an ordinary differential equation plus an application of Girsanov's theorem to adjust the drift. The discrete approximation is based on a specific strong approximation of Brownian motion by simple, symmetric random walks (the so-called ``twist and shrink'' method).  A discrete It\^o's formula is also used during the discrete approximation.
\end{abstract}


\renewcommand{\thefootnote}{\alph{footnote}}
\footnotetext{ 2010 \emph{MSC.} Primary 60H10. Secondary 60H35, 60F15.} 


\section{Introduction} \label{sec:Intro}

Stochastic differential equations (SDE's) have been widely applied to study systems modeled by differential equations that include random effects. Such systems appear in a wide spectrum of fields, like financial mathematics, biological models, physical systems, etc.

Then it is a natural demand to find numerical approximations of a solution that cannot be determined exactly, just like in the case of non-stochastic differential equations, see e.g. \cite[Ch. VI, Section 7]{IW1981} and \cite[Ch. 5, Section 5.2.D]{KS1999}.

The aim of this paper is to develop a sequence of discrete approximations to a one-dimensional It\^o diffusion that almost surely converges to a weak solution of the given SDE. First, the solution of the SDE is reduced to the solution of an ordinary differential equation (ODE), plus an application of Girsanov's theorem to adjust the drift, see \cite{Hoek2009}. Second, we use a discrete approximation which is based on a specific strong approximation of Brownian motion by simple, symmetric random walks (the so-called ``twist and shrink'' method) and a discrete It\^o's formula, see \cite{Szab1996}. The bottle-neck of our proposed method is the application of Girsanov's theorem to adjust the drift: it requires that the corresponding Radon--Nikodym derivative be a martingale.

To specify our aim more exactly, consider a one-dimensional \emph{stochastic differential equation (SDE)}
\begin{eqnarray} \label{eq:SDE}
dX(t) &=& \mu(t,X(t)) \, dt + \sigma(t,X(t)) \, dB(t) \\
X(0) &=& x_0, \qquad 0 \le t \le T < \infty, \nonumber
\end{eqnarray}
where $B$ is Brownian motion on a complete probability space $(\Omega, \mathcal{F}, \mathbb{P})$, $B(0)=0$, $x_0 \in \mathbb{R}$, and $\mu, \sigma : [0, T]\times \mathbb{R} \to \mathbb{R}$ are the \emph{drift} and \emph{diffusion coefficients}, respectively. A solution $X$ is called an \emph{It\^o diffusion}.

We want to find a sequence of simple, symmetric random walks $(B_m(t))_{t \ge 0}$, $m=0,1,2, \dots$ and a deterministic function $\phi(t,x)$ -- which will be a solution of an ordinary differential equation adjoined to the SDE -- such that
\[
\lim_{m \to \infty} \sup_{0 \le t \le T} |B_m(t) - B(t)| = 0 \qquad \text{a.s.}
\]
and with $X_m(t) := \phi(t, B_m(t))$,
\[
\lim_{m \to \infty} \sup_{0 \le t \le T} |X_m(t) - X(t)| = 0 \qquad \text{a.s.},
\]
where $X(t) := \phi(t, B(t))$ is a \emph{weak} solution of the SDE.


\section{A method of finding a weak solution of an SDE} \label{sec:Meth}

First let us suppose that the ordinary differential equation (ODE):
\begin{equation}\label{eq:ODE}
\phi'_u(t,u) = \sigma(t, \phi(t,u)), \qquad \phi(t,0)=x_0
\end{equation}
has a solution over all values of $u \in \mathbb{R}$, with each value of the parameter $t \in [0,T]$, and its solution $\phi(t,u)$ is a $C^{1,2}$ function. (Lemma \ref{le:suffice} below gives a sufficient condition.)

Then define the stochastic process
\begin{equation}\label{eq:diff1}
X(t) := \phi(t, B(t)), \qquad 0 \le t \le T.
\end{equation}
By It\^o's formula,
\begin{equation}\label{eq:diff2}
dX(t) = \nu(t, B(t)) dt + \sigma(t, X(t)) dB(t),
\end{equation}
where
\begin{equation}\label{eq:nu}
\nu(t, B(t)) := \phi'_t(t, B(t)) + \frac12 \phi''_{uu}(t, B(t)),
\end{equation}
and $X(0)=x_0$. Thus $X$ is an \emph{It\^o process} with the correct diffusion coefficient and initial value, but its drift coefficient is not the one we wanted.

We now make a change of probability to adjust the drift in (\ref{eq:diff2}) to coincide with that in (\ref{eq:SDE}). By Girsanov's theorem, this can be achieved by introducing a new probability measure $\mathbb{Q}$ by setting
\begin{equation} \label{eq:Q}
\frac{d\mathbb{Q}}{d\mathbb{P}} = \Lambda(T)
:= \exp\left\{-\int_0^T \psi(s) dB(s) -\frac12 \int_0^T \psi^2(s) ds \right\},
\end{equation}
where
\begin{eqnarray}
\psi(t) &:=& \frac{\nu(t, B(t)) - \mu(t, X(t))}{\sigma(t, X(t))} =  \widetilde{\psi}(t, B(t)), \label{eq:psi} \\
\widetilde{\psi}(t, u) &:=& \frac{\phi'_t(t, u) + \frac12 \phi''_{uu}(t, u) - \mu(t,\phi(t,u))}{\sigma(t,\phi(t,u))} . \label{eq:psitilde}
\end{eqnarray}
To apply Girsanov's theorem, we have to assume that the process
\begin{equation}\label{eq:Lambda}
\Lambda(t) := \exp\left\{-\int_0^t \psi(s) dB(s) -\frac12 \int_0^t \psi^2(s) ds \right\} \qquad (0 \le t \le T)
\end{equation}
is a $\mathbb{P}$-martingale. Then
\begin{equation}\label{eq:W}
W(t) := B(t) + \int_0^t \psi(s) ds
\end{equation}
is a $\mathbb{Q}$-Brownian motion, and the process
\[
X(t) = \phi(t, B(t)) = \phi\left(t, W(t) - \int_0^t \psi(s) ds \right)
\]
is a solution of the SDE
\begin{equation} \label{eq:SDEQ}
dX(t) = \mu(t,X(t)) \, dt + \sigma(t,X(t)) \, dW(t), \qquad X(0) = x_0,
\end{equation}
on $[0, T]$. It means that $X$ is a weak solution of (\ref{eq:SDE}).

The next lemma gives a sufficient condition under which our method works and it uniquely leads to a weak solution $X(t)= \phi(t, B(t))$ $(0 \le t \le T)$.
\begin{lem} \label{le:suffice}
Consider the following three conditions: 
\begin{enumerate}[(i)]
  \item $\sigma \in C^{1,1}\left([0, T]\times \mathbb{R}\right)$ and $|\sigma(t,x)| \le K_0(1 + |x-x_0|)$ for any $(t,x) \in [0, T] \times \mathbb{R}$, where $K_0$ is a finite constant;
  \item $\mu(t,x)$ is continuous and  $\sigma(t,x) > 0$ for any $(t,x) \in [0, T]\times D$, where $D$ denotes the range of the solution $\phi$ over $[0, T]\times \mathbb{R}$;
  \item $|\widetilde{\psi}(t, u)| \le K (1 + |u|)$ for any $(t,u) \in [0, T] \times \mathbb{R}$, where $\widetilde{\psi}$ is defined by (\ref{eq:psitilde}) and $K$ is a finite constant.
\end{enumerate}

Then we have the following claims: 
\begin{enumerate}[(a)]
\item Under condition (i), the ODE (\ref{eq:ODE}) is uniquely solvable for any $(t,u) \in [0, T] \times \mathbb{R}$ and the solution $\phi$ belongs to $C^{1,2}\left([0, T]\times \mathbb{R}\right)$.
    
\item If in addition we assume conditions (ii) and (iii), then $\widetilde{\psi} \in C\left([0, T]\times \mathbb{R}\right)$ and the process $\Lambda(t)$ given in (\ref{eq:Lambda}) is a $\mathbb{P}$-martingale for $t \in [0, T]$.
\end{enumerate}
\end{lem}
\begin{proof}
\emph{(a)} Because of symmetry, it is enough to consider the case when $u \ge 0$. Fix an arbitrary $t\in [0, T]$ and take first the rectangle $R_1=[0,h_1]\times[x_0-1, x_0+1]$ in the $(u,x)$ plane, where $h_1=1/(2K_0)$. By the linear growth assumption in \emph{(i)}, $|\sigma(t,x)| \le 2K_0$ on $R_1$, so by classical theorems, see e.g. \cite{Hur1958}, the initial value problem (\ref{eq:ODE}) has a unique solution $\phi(t,u)$ for $u \in [0,h_1]$ and its graph belongs to $R_1$.

Let $x_1:=\phi(t,h_1)$. Then take the rectangle $R_2 := [h_1, h_1+h_2]\times[x_0-2, x_0+2]$, where $h_2 = 1/(3K_0)$. Again, by the linear growth in \emph{(i)}, $|\sigma(t,x)| \le 3K_0$ on $R_2$, so the initial value problem (\ref{eq:ODE}) with new initial condition $\phi(t,h_1) = x_1$ has a unique solution $\phi(t,u)$ for $u \in [h_1,h_1+h_2]$ and its graph belongs to $R_2$.

Continuing this way, we get a sequence of adjoining rectangles $R_1, R_2, \dots$ with total horizontal length $\sum_{j=2}^{\infty} 1/(jK_0) = \infty$, and so a unique solution of (\ref{eq:ODE}) is obtained for all values $u \ge 0$.

We supposed that $\sigma(t,x)$ has continuous partials with respect to the variables $t$ and $x$ over $[0, T]\times\mathbb{R}$. Hence by \cite[Ch. 2, Theorem 10]{Hur1958} it follows that the solution $\phi(t,u)$ has a continuous partial derivative with respect to $t$ over $[0, T]\times\mathbb{R}$.

Differentiating (\ref{eq:ODE}) with respect to $u$, it also follows that
\[
\phi''_{uu}(t,u) = \sigma'_x(t, \phi(t,u)) \, \phi'_u(t,u) = \sigma'_x(t, \phi(t,u)) \, \sigma(t,\phi(t,u))
\]
exists and is continuous on $[0, T]\times\mathbb{R}$. This completes the proof of (a).

\emph{(b)} Conditions \emph{(i)} and \emph{(ii)} imply the continuity of $\widetilde{\psi}$. Condition \emph{(iii)} implies Bene\v{s}' condition, which in turn is a weakened version of Novikov's condition, see e.g. \cite[Ch. 3, 5.16 Corollary]{KS1999}.
\end{proof}


\section{Preliminaries of a discrete approximation} \label{sec:Pre}

\subsection{The ``twist and shrink'' approximation of Brownian motion}

A basic tool of the present paper is an elementary construction of Brownian motion. The specific
construction used in the sequel, taken from \cite{Szab1996}, is based on a nested sequence of simple,
symmetric random walks that uniformly converges to Brownian motion (BM = Wiener process) on bounded intervals with probability $1$. This will be called \emph{``twist and shrink''} construction. This method is a modification of the one given by Frank Knight in 1962 \cite{Kni1962}.

We summarize the major steps of the ``twist and shrink'' construction here. We start with \emph{a sequence of independent simple, symmetric random walks} (abbreviated: RW)
\[
S_m(0) = 0, \quad S_m(n) = \sum_{k=1}^{n} X_m(k) \quad (n \ge 1),
\]
based on an infinite matrix of independent and identically distributed random variables $X_m(k)$,
\[
\mathbb{P} \left\{ X_m(k)= \pm 1 \right\} = \frac12 \qquad (m\ge 0, k\ge 1),
\]
defined on the same complete probability space $(\Omega,\mathcal{F},\mathbb{P})$. (All stochastic processes in the sequel will be defined on this probability space.) Each random walk is a basis of an approximation of Brownian motion with a dyadic step size $\Delta t=2^{-2m}$ in time and a corresponding step size $\Delta x=2^{-m}$ in space.

The second step of the construction is \emph{twisting}. From the independent RW's we want to create
dependent ones so that after shrinking temporal and spatial step sizes, each consecutive RW becomes a
refinement of the previous one.  Since the spatial unit will be halved at each consecutive row, we
define stopping times by $T_m(0)=0$, and for $k\ge 0$,
\[
T_m(k+1)=\min \{n: n>T_m(k), |S_m(n)-S_m(T_m(k))|=2\} \qquad (m\ge 1)
\]
These are the random time instants when a RW visits even integers, different from the previous one.
After shrinking the spatial unit by half, a suitable modification of this RW will visit the same
integers in the same order as the previous RW. In other words, if $\widetilde{S}_{m-1}$ visits the integers $i_0=0, i_1, i_2, i_3, \dots$, $(i_j \ne i_{j+1})$, then we want that the twisted random walk $\widetilde{S}_m$ visit the even integers $2i_0, 2i_1, 2i_2, 2i_3$ in this order.

We operate here on each point $\omega\in\Omega$ of the
sample space separately, i.e. we fix a sample path of each RW. We define twisted RW's $\widetilde{S}_m$
recursively for $k=1,2,\dots$ using $\widetilde{S}_{m-1}$, starting with $\widetilde{S}_0(n)=S_0(n)$ $(n\ge
0)$ and $\widetilde{S}_m(0) = 0$ for any $m \ge 0$. With each fixed $m$ we proceed for $k=0,1,2,\dots$
successively, and for every $n$ in the corresponding bridge, $T_m(k)<n\le T_m(k+1)$. Each bridge is
flipped if its sign differs from the desired:
$\widetilde{X}_m(n) = \pm  X_m(n)$, depending on whether $S_m(T_m(k+1)) - S_m(T_m(k))
= 2\widetilde X_{m-1}(k+1)$ or not. So $\widetilde{S}_m(n)=\widetilde{S}_m(n-1)+\widetilde{X}_m(n)$.

Then $(\widetilde{S}_m(n))_{n\ge 0}$ is still a
simple symmetric RW \cite[Lemma 1]{Szab1996}. The twisted RW's have the desired refinement property:
\[
\widetilde{S}_{m+1}(T_{m+1}(k)) = 2 \widetilde{S}_{m}(k) \qquad (m\ge 0, k\ge 0).
\]

The third step of the RW construction is \emph{shrinking}. The sample paths of $\widetilde{S}_m(n)$ $(n\ge
0)$ can be extended to continuous functions by linear interpolation, this way one gets
$\widetilde{S}_m(t)$ $(t\ge 0)$ for real $t$. The $mth$ \emph{``twist and shrink'' RW} is defined by
\begin{equation}\label{eq:TAS}
\widetilde{B}_m(t)=2^{-m}\widetilde{S}_m(t2^{2m}).
\end{equation}
Then the \emph{refinement property} takes the form
\begin{equation}
\widetilde{B}_{m+1}\left(T_{m+1}(k)2^{-2(m+1)}\right) = \widetilde{B}_m \left( k2^{-2m}\right) \qquad (m\ge
0,k\ge 0). \label{eq:refin}
\end{equation}
Note that a refinement takes the same dyadic values in the same order as the previous shrunken walk,
but there is a \emph{time lag} in general:
\begin{equation} T_{m+1}(k)2^{-2(m+1)} - k2^{-2m} \ne 0 .
\label{eq:tlag}
\end{equation}

Now let us recall some important facts from \cite{Szab1996} and \cite{SzaSze2009} about the ``twist
and shrink'' construction that will be used in the sequel.

\begin{thmA} \label{th:Wiener}
The sequence of ``twist and shrink'' random walks ˜$B_m$ uniformly converges to Brownian motion $B$ on bounded intervals, almost surely. For all $T > 0$ fixed, as $m \to \infty$,
\[
\sup_{0 \le t \le T} |B(t) - \widetilde{B}_m(t)| = O\left(m^{\frac34} 2^{-\frac{m}{2}}\right) \qquad \text{a.s.}
\]
\end{thmA}

Conversely, with a given Brownian motion $B$, one can define the stopping times which yield the
\emph{Skorohod embedded RW's} $B_m(k2^{-2m})$ into $B$. For every $m\ge 0$ let $s_m(0)=0$ and
\begin{equation} \label{eq:Skor1}
s_m(k+1)=\inf{}\{s: s > s_m(k), |B(s)-B(s_m(k))|=2^{-m}\} \qquad (k \ge 0).
\end{equation}
With these stopping times the embedded dyadic walks by definition are
\begin{equation} \label{eq:Skor2}
B_m(k2^{-2m}) = B(s_m(k)) \qquad (m\ge 0, k\ge 0).
\end{equation}
This definition of $B_m$ can be extended to any real $t \ge 0$ by pathwise linear interpolation.

If Brownian motion is built by the ``twist and shrink'' construction described above using a sequence
$(\widetilde{B}_m)$ of nested RW's and then one constructs the Skorohod embedded RW's $(B_m)$, it is
natural to ask about their relationship. It is important that they are asymptotically
equivalent, so Theorem \ref{th:Wiener} is valid for $(B_m)$ as well, in any dimension $d$, cf. \cite{Szab1996} and \cite{Szab2012}. For all $T > 0$ fixed, as $m \to \infty$,
\begin{equation}\label{eq:Skor}
\sup_{0 \le t \le T} |B(t) - B_m(t)| = O\left(m^{\frac34} 2^{-\frac{m}{2}}\right) \qquad \text{a.s.}
\end{equation}

In general, $(\widetilde{B}_m)$ is more useful when someone wants to generate
stochastic processes from scratch, while $(B_m)$ is more advantageous when someone needs discrete
approximations of given processes.


\subsection{A discrete It\^o's formula}

The second major tool used in this paper is \emph{a discrete It\^o's formula}. It is interesting that one can give discrete versions of It\^o's formulas which are purely algebraic identities, not assigning any probabilities to the terms. Despite
this, the usual It\^o's formulas follow fairly easily from the discrete counterparts in a proper probability setting. Apparently, the first such formula was given by Kudzma in 1982
\cite{Kud1982}. The elementary algebraic approach used in the present paper is different from that; it
was introduced by the second author in 1989 \cite{Szab1990}.

Fix an initial point $a \in \mathbb{R}$ and step-size (mesh) $h > 0$. Consider the grid $\mathcal{G}(a,h)
:= a + h \mathbb{Z}$, and let $f : \mathcal{G}(a,h) \rightarrow \mathbb{R}$ be a function on this grid. Take an arbitrary broken line (\emph{a discrete path}) $\gamma $ that goes through finitely many (not
necessarily distinct) oriented edges between adjoining vertices of the grid. A typical such edge is
$[x, x + \mu h]$, where $x \in \mathcal{G}(a,h)$ and $\mu = \pm 1$. (The order of the two vertices is important!) A discrete path $\gamma $ is a formal sum of such oriented edges (that is, a 1-chain): $\gamma = \sum_{r=0}^{n-1} [x_r, x_r + \mu_r h]$.

By definition, the corresponding \emph{discrete path integral} or \emph{trapezoidal sum} of $f$ over
$\gamma $ is defined as
\begin{multline*}
T_{\gamma } \, f \, h := \frac{h}{2} \sum_{r=1}^n \mu_r  \left(f(x_r) +
f(x_r + \mu_r h)\right) \\
= h \, \mathrm{sgn}(x_n-x_0) \left\{\frac12 f(x_0) + \frac12 f(x_n) + \sum_{j=1}^{|x_n - x_0|/h - 1} f\left(x_0 + j \, h \, \mathrm{sgn}(x_n - x_0)\right)\right\} \\
=: T_{x=x_0}^{x_{n}} f(x) h.
\end{multline*}

The above definition of a trapezoidal sum shows that the orientation of an edge is defined by the
order of its two vertices: it is positive if the edge goes increasingly and negative
in the opposite case. If $\gamma = \emptyset$ (or $x_0=x_{n}$), we define $T_{\gamma } f h = 0$. It is also clear that, like in the case of a conservative vector field, the sum depends only on the initial and end points, does not otherwise depend on the path.

Versions of the following discrete It\^o's formula (which is a simple algebraic identity) already appeared in \cite[Proposition 3]{Szab1990}, \cite[Lemma 11]{Szab1996} and \cite[Lemma 1]{Szab2012}.

\begin{lemA}\label{le:discrete_Ito}
Take $a \in \mathbb{R}$, step $h > 0$, and a time-dependent function $f : h^2 \mathbb{Z}_+ \times \mathcal{G}(a,h) \rightarrow \mathbb{R}$. Consider a
sequence $(\xi_r)_{r \ge 1}$, where $\xi_{r} = \pm 1$. Define partial sums
$S_0=a $, $S_n = a+h(\xi_1+\cdots +\xi_n)$ ($n \ge 1$) and discrete time instants $t_r = r h^2$ ($0 \le r
\le n$). Assume that the steps of $(S_n)$ are performed in time steps $h^2$. Then the following
equalities hold:
\begin{multline} \label{eq:disc_Strat}
T_{x=S_0}^{S_n} f(t_n,x) h
= \sum_{r=1}^{n}  T_{x=S_0}^{S_r} \frac{f(t_r,x) - f(t_{r-1},x)}{h^2} h h^2  \\
+ \sum_{r=1}^{n} \frac{f \left(t_{r-1}, S_{r-1} \right) + f \left(t_{r-1}, S_{r} \right)}{2} h \xi_r
\end{multline}
(discrete Stratonovich formula). Alternatively,
\begin{multline} \label{eq:disc_Ito}
T_{x=S_0}^{S_n} f(t_n,x) h
= \sum_{r=1}^{n}  T_{x=S_0}^{S_r} \frac{f(t_r,x) - f(t_{r-1},x)}{h^2} h h^2   \\
+ \sum_{r=1}^{n} f \left(t_{r-1}, S_{r-1} \right) h \xi_r
+  \frac12 \sum_{r=1}^n \frac{f \left(t_{r-1}, S_{r} \right) - f \left(t_{r-1}, S_{r-1} \right)} {h \xi_r} h^2
\end{multline}
(discrete It\^o's formula).
\end{lemA}


Let us apply now the discrete It\^o's formula (\ref{eq:disc_Ito}) to (the $x$ partial of) a
random, time-dependent function $g : \Omega \times \mathbb{R}_+ \times \mathbb{R} \to
\mathbb{R}$, $g(\omega,t,x)$, which is measurable in $\omega$ for all $(t,x)$, and is $C^{1,2}$ in
$(t,x)$ for almost all $\omega $.

Start with a Brownian motion $(B(t))_{t \ge 0}$ shifted so that $B(0)=a$. Then take Skorohod embedded random walks
$(B_m(t))_{t \ge 0}$, $B_m(r2^{-2m})=B(s_m(r))$ in (\ref{eq:disc_Ito}).
That is, let $S_r:=B_m(r2^{-2m})$ and
\[
\xi_r = \xi_m(r) := 2^m \left\{B_m(r2^{-2m}) - B_m((r-1)2^{-2m}) \right\} .
\]
Then $(\xi_m(r))_{r=1}^{\infty}$ is an independent, $(\pm 1)$, symmetric coin
tossing sequence. Define \emph{stochastic sums} by
\begin{equation} \label{eq:stocsum1}
\left(f(\omega, s, B) \cdot B \right)^m_t := \sum_{r=1}^{n} f\left(\omega , t_{r-1},
B_m(t_{r-1})\right) \: 2^{-m} \xi_m(r) ,
\end{equation}
where $f=g'_x$, $t_r := r 2^{-2m}$ and $n := \lfloor t 2^{2m} \rfloor$. (Of course, $B_m$, $\xi_m$, and $B$ all depend on $\omega $, but this dependence is not shown here and below, to simplify the notation.)

Now the discrete It\^o's formula (\ref{eq:disc_Ito}) can be written as
\begin{eqnarray} \label{eq:disc_Ito1}
\lefteqn{T_{x=a}^{B_m(t_{n})} f(\omega , t_{n}, x) 2^{-m}}
\\
& = & \sum_{r=1}^{n} T_{x=a}^{B_m(t_r)} \frac{f(\omega , t_r, x)
- f(\omega , t_{r-1}, x)}{2^{-2m}} 2^{-m} 2^{-2m}  \nonumber \\
& + & \sum_{r=1}^{n} f\left(\omega , t_{r-1},
B_m(t_{r-1})\right) \: 2^{-m} \xi_m(r) \nonumber \\
& + &  \frac12 \sum_{r=1}^{n} \frac{ f\left(\omega , t_{r-1}, B_m(t_r)\right) - f\left(\omega , t_{r-1}, B_m(t_{r-1}) \right)}
{2^{-m} \xi_m(r)} 2^{-2m} . \nonumber
\end{eqnarray}

One can show, cf. \cite[Theorem 6]{Szab1996} and \cite[Theorem 1]{Szab2012}, that each term in this formula almost surely uniformly converges to the corresponding term of the It\^o's formula, on any bounded
time interval. In particular, the stochastic sum almost surely uniformly converges to
the stochastic integral, on any bounded time interval.

\begin{thmA} \label{th:Ito}
Suppose $g(\omega,t,x)$ is measurable in $\omega$ for all $(t,x)$, and is $C^{1,2}$ in $(t,x)$ for
almost every $\omega $. Let $f=g'_x$. Taking Brownian motion $B$, for each $m$ define the Skorohod
embedded random walk $B_m$. Then for arbitrary $T>0$,
\[
\sup_{t\in[0,T]}\left|\left(f(\omega, s, B) \cdot B \right)_t^m - \int_0^t f(\omega , s,
B(s)) dB(s) \right| \rightarrow 0
\]
almost surely as $m \to \infty$, and for any $t \ge 0$ we obtain the It\^o's formula as an
almost sure uniform limit on any bounded interval $[0,T]$ of the discrete formula (\ref{eq:disc_Ito}), term-by-term:
\begin{eqnarray} \label{eq:Ito_form}
\lefteqn{g(\omega , t, B(t)) - g(\omega , 0, B(0)) = \int_{0}^{t}
g'_t(\omega , s, B(s)) ds} \\
& & + \int_{0}^{t} g'_x(\omega , s, B(s)) dB(s) + \frac12 \int_{0}^{t} g''_{xx}(\omega , s, B(s)) ds . \nonumber
\end{eqnarray}
\end{thmA}

The reader may have noticed in the statement of Theorem \ref{th:Ito} that the usual condition
in It\^o's formulae that the random function $g(\omega , t, x)$ be adapted to the filtration of
Brownian motion $B$ in the variable $\omega$, was not needed: the assumed smoothness of $g$ together with the pathwise, integration by parts stochastic integration technique made this assumption unnecessary.


\section{A sequence of discrete approximations of a diffusion} \label{sec:Disc}

Based on Section \ref{sec:Meth}, we consider the following three assumptions:

\emph{Assumption 1:} $\mu$ and $\sigma$ are continuous and $\sigma > 0$ on $[0, T]\times D$, where $D$ denotes the range of the unique solution $\phi$ of ODE (\ref{eq:ODE}).

\emph{Assumption 2:} $\phi \in C^{1,2}([0, T] \times \mathbb{R})$.

\emph{Assumption 3:} $\Lambda(t)$ is a $\mathbb{P}$-martingale for $t \in [0, T]$.

\medskip

Assumptions 1 and 2 will always be used below, while we will need Assumption 3 as well in several instances.  
Then $X(t) = \phi(t, B(t))$, $(0 \le t\le T)$ is well-defined and is a weak solution of the SDE (\ref{eq:SDE}). Lemma \ref{le:suffice} above contain a sufficient condition for this.

Our task is now twofold: first we introduce discrete approximations of $X(t)$, then we discuss approximations of the probability measure $\mathbb{Q}$ defined by (\ref{eq:Q}).
So first we introduce a sequence of approximations, indexed by $m=0,1,2,\dots$, by replacing Brownian motion $B$ with its Skorohod embedded random walks $B_m$:
\begin{equation} \label{eq:X_m}
X_m(t) := \phi(t, B_m(t)) \qquad (0 \le t \le T, m=0,1,2,\dots).
\end{equation}
(Though we use continuous time, remember that $B_m$ is essentially a discrete random walk. We interpolated it piecewise linearly in time for sake of convenience.)

\begin{thm} \label{th:Xconv}
Under Assumptions 1 and 2 above, let $X(t) = \phi(t, B(t))$, $(0 \le t\le T)$ be the weak solution of the SDE (\ref{eq:SDE}) as discussed in Section \ref{sec:Meth} and take the approximations $X_m(t)$ defined by (\ref{eq:X_m}).
Then
\[
\sup_{0 \le t \le T} |X_m(t) - X(t)| = O\left(m^{\frac34} 2^{-\frac{m}{2}}\right) \qquad \text{a.s.}
\]
\end{thm}
\begin{proof}
Fix an $\omega$ in the probability 1 subset where (\ref{eq:Skor}) holds. Since $B(\omega, t)$ is continuous on $[0,T]$, its range is compact. Also, by the uniform convergence in (\ref{eq:Skor}), there is a compact set $K_{\omega} \subset \mathbb{R}$ that contains the range of $B(\omega, t)$ and that of each $B_m(\omega, t)$ over $[0, T]$, $(m=0,1,2,\dots)$. Take $L_{\omega} := \sup_{(t,u) \in [0,T]\times K_{\omega}} |\phi'_u(t,u)| < \infty$. Then
\begin{multline*}
\sup_{0 \le t \le T} |X_m(\omega, t) - X(\omega, t)| = \sup_{0 \le t \le T} |\phi(t, B_m(\omega, t)) - \phi(t,B(\omega, t))| \\
\le L_{\omega} \, |B_m(\omega, t) - B(\omega, t)| =  O\left(m^{\frac34} 2^{-\frac{m}{2}}\right) \end{multline*}
by (\ref{eq:Skor}).
\end{proof}

Now under Assumptions 1,2 and 3 we are going to determine a sequence of discrete approximations of the probability measure $\mathbb{Q}$, defined by(\ref{eq:Q}). For this, we use the drift term in a discrete It\^o's formula. Define
\begin{equation}\label{eq:appr1}
X_m^*(t_r) := x_0 + T_{u=0}^{B_m(t_r)} \phi'_u(t_r, u) 2^{-m} = X_m(t_r) + O(2^{-m}) ,
\end{equation}
where $t_r = r 2^{-2m} \in [0,T]$. The second equality follows from (\ref{eq:Skor}) and from the continuity of $\phi'_u$, see Assumption 2 above. Note that the error term $O(2^{-m})$ may depend on $\omega$.

Apply the discrete It\^o's formula (\ref{eq:disc_Ito1}) to $X_m^*$ with $f:=\phi'_u$:
\begin{multline}\label{eq:disc_Ito2}
X_m^*(t_n) - x_0 = T_{u=0}^{B_m(t_n)} \phi'_u(t_n, u) 2^{-m} \\
=  \sum_{r=1}^{n} \phi'_u\left(t_{r-1}, B_m(t_{r-1})\right) \: \left(B_m(t_r) - B_m(t_{r-1})\right)  \\
+  \sum_{r=1}^{n} \left\{T_{u=0}^{B_m(t_r)} \frac{\phi'_u(t_r, u)
- \phi'_u(t_{r-1}, u)}{2^{-2m}} 2^{-m} \right. \\
\left. +  \frac12 \frac{ \phi'_u\left(t_{r-1}, B_m(t_r)\right) - \phi'_u\left(t_{r-1}, B_m(t_{r-1}) \right)} {B_m(t_r) - B_m(t_{r-1})} \right\} 2^{-2m} .
\end{multline}
By Theorem \ref{th:Ito}, the terms of this formula a.s. converge, uniformly on $[0,T]$, to the corresponding terms in the continuous It\^o's formula as $m \to \infty$.

In particular, the stochastic sum on the right hand side tends to the stochastic integral, that is, to the diffusion term as $m \to \infty$:
\begin{multline}\label{eq:appr3}
\sum_{r=1}^{\lfloor t 2^{2m}\rfloor} \phi'_u\left(t_{r-1}, B_m(t_{r-1})\right) \: \left(B_m(t_r) - B_m(t_{r-1})\right) \to \int_{0}^{t} \phi'_u(s, B(s))  dB(s) \\
=  \int_{0}^{t} \sigma(s, X(s)) dB(s),
\end{multline}
since by (\ref{eq:ODE}), $\phi'_u(s, B(s)) = \sigma(s, \phi(s, B(s))) = \sigma(s, X(s))$.

It is important that the last sum on the right hand side tends to the drift term, that is, to the integral of $\nu(t, B(t))$ in (\ref{eq:diff2}) as $m \to \infty$:
\begin{multline}\label{eq:appr4}
\sum_{r=1}^{\lfloor t 2^{2m}\rfloor} \nu_m(t_{r}, B_m(t_{r-1}), B_m(t_r)) \, 2^{-2m} \\
:= \sum_{r=1}^{\lfloor t 2^{2m}\rfloor} \left\{T_{u=0}^{B_m(t_r)} \frac{\phi'_u(t_r, u)
- \phi'_u(t_{r-1}, u)}{2^{-2m}} 2^{-m} \right. \\
\left. +  \frac12 \frac{ \phi'_u\left(t_{r-1}, B_m(t_r)\right) - \phi'_u\left(t_{r-1}, B_m(t_{r-1}) \right)} {B_m(t_r) - B_m(t_{r-1})} \right\} 2^{-2m} \\
\longrightarrow \int_{0}^{t} \nu(s, B(s)) ds = \int_{0}^{t} \left\{\phi'_t(s, B(s)) + \frac12 \phi''_{uu}(s, B(s)) \right\} ds
\end{multline}

The next lemma gives a local version of (\ref{eq:appr4}).
\begin{lem} \label{le:locdrift}
In addition to Assumptions 1 and 2 above, let us assume that $\sigma(t,x) \in C^{1,1}([0, T] \times \mathbb{R})$. Then almost surely, as $m \to \infty$, we have 
\begin{multline} \label{eq:nu_m1}
\sup_{0 \le t_r \le T} |\nu_m(t_r, B_m(t_{r-1}), B_m(t_r)) - \nu(t_r, B_m(t_r))| \\
:= \sup_{0 \le t_r \le T} \left| T_{u=0}^{B_m(t_r)} \frac{\phi'_u(t_r, u) - \phi'_u(t_{r-1}, u)}{2^{-2m}} 2^{-m} \right. \\
\left. +  \frac12 \frac{ \phi'_u\left(t_{r-1}, B_m(t_r)\right) - \phi'_u\left(t_{r-1}, B_m(t_{r-1}) \right)} {B_m(t_r) - B_m(t_{r-1})}   \right. \\
\left. - \left(\phi'_t(t_r, B_m(t_r)) + \frac12 \phi''_{uu}(t_r, B_m(t_r))\right) \right| \to 0,
\end{multline}
where by definition, $\nu_m(t_0, B_m(t_{-1}), B_m(t_0)):= \frac12 \phi''_{uu}(0, 0)$.

Moreover, with $t^m := \lfloor t 2^{2m} \rfloor 2^{-2m}$,
\begin{multline}\label{eq:nu_m}
\lim_{m \to \infty} \nu_m(t^m, B_m(t^m-2^{-2m}), B_m(t^m)) = \lim_{m \to \infty} \nu(t^m, B_m(t^m)) \\
= \lim_{m \to \infty} \nu(t, B_m(t)) =  \nu(t, B(t)),
\end{multline}
a.s., uniformly over $t \in [0, T]$.
\end{lem}
\begin{proof}
By formula (\ref{eq:Skor}), there exists an $\Omega_1 \subset \Omega$ such that $\mathbb{P}(\Omega_1)=1$ and for every $\omega \in \Omega_1$,
\[
\sup_{0 \le t \le T} |B_m(\omega, t) - B(\omega, t)| \le C_1 \, m^{\frac34} 2^{-\frac{m}{2}}
\]
when $m \ge m_0(\omega)$  and the path $t \mapsto B(\omega, t)$ is continuous. Fix such an $\omega \in \Omega_1$ in the sequel. Then there exists a compact set $K(\omega) \subset \mathbb{R}$ that contains the range of $B(\omega,t)$ and of all $B_m(\omega, t)$ as $m \ge m_0(\omega)$ over $t \in [0, T]$. Define $C_2(\omega) := \sup \{ |x| : x \in K(\omega) \}$.

We need the fact that under the assumptions of this lemma, $\phi''_{ut}$ exists and is continuous, so $\phi''_{tu} = \phi''_{ut}$ exists as well. For, by (\ref{eq:ODE}),
\[
\phi''_{ut}(t, u) = \sigma'_t(t, \phi(t,u)) + \sigma'_x(t, \phi(t, u) \, \phi'_t(t,u) .
\]

Next, by the mean value theorem, with each $t_r \in [t_1, T]$ and $u \in K(\omega)$ there is a number $s_r \in [t_{r-1}, t_r]$ such that
\[
\frac{\phi'_u(t_r, u) - \phi'_u(t_{r-1}, u)}{2^{-2m}} = \phi''_{tu}(s_r, u).
\]
By our assumptions $\phi''_{tu}$ is continuous, so it is uniformly continuous on $[0,T] \times K(\omega)$. Thus for any $\epsilon >0$, there exists $m_1(\omega)$ such that for any $m \ge m_1(\omega)$ we have
\[
|\phi''_{tu}(s_r,u) - \phi''_{tu}(t_r,u)| < \frac{\epsilon}{4 C_2(\omega)}
\]
(since $|s_r - t_r| \le 2^{-2m}$) and
\[
|\phi''_{tu}(t_r,u) - \phi''_{tu}(t_r,\tilde{u})| < \frac{\epsilon}{4 C_2(\omega)},
\]
when $t_r \in[t_1, T]$, $u, \tilde{u} \in K(\omega)$, $|u - \tilde{u}| \le 2^{-m}$.

These imply that
\[
\left| T_{u=0}^{B_m(\omega, t_r)} \phi''_{tu}(s_r, u) 2^{-m} - T_{u=0}^{B_m(\omega, t_r)} \phi''_{tu}(t_r, u) 2^{-m} \right| < \frac{\epsilon}{4 C_2(\omega)} |B_m(\omega, t_r)| \le \frac{\epsilon}{4}
\]
and
\[
\left| T_{u=0}^{B_m(\omega, t_r)} \phi''_{tu}(t_r, u) 2^{-m} - \int_{0}^{B_m(\omega, t_r)} \phi''_{tu}(t_r, u) d u \right| < \frac{\epsilon}{4 C_2(\omega)} |B_m(\omega, t_r)| \le \frac{\epsilon}{4}
\]
for each $t_r \in [t_1, T]$, when $m \ge \max\{m_0(\omega), m_1(\omega)\}$.

Since
\[
\int_{0}^{B_m(\omega, t_r)} \phi''_{tu}(t_r, u) d u = \phi'_t(t_r, B_m(\omega, t_r)) - \phi'_t(t_r, 0)
\]
and $\phi'_t(t_r, 0) = 0$ by the condition $\phi(t,0) = x_0$ ($t \in [0, T]$) in (\ref{eq:ODE}), it follows that
\begin{equation}\label{eq:locdrift1}
\left| T_{u=0}^{B_m(\omega, t_r)} \frac{\phi'_u(t_r, u) - \phi'_u(t_{r-1}, u)}{2^{-2m}} 2^{-m} - \phi'_t(t_r, B_m(\omega, t_r)) \right| < \frac{\epsilon}{2}
\end{equation}
for each $t_r \in [t_1, T]$, when $m \ge \max\{m_0(\omega), m_1(\omega)\}$.

Further, by the mean value theorem, with each $t_r \in [t_1, T]$ there is a number $\tilde{s}_r \in [t_{r-1}, t_r]$ such that
\begin{equation}\label{eq:locdrift2}
\frac{ \phi'_u\left(t_{r-1}, B_m(\omega, t_r)\right) - \phi'_u\left(t_{r-1}, B_m(\omega, t_{r-1}) \right)} {B_m(\omega, t_r) - B_m(\omega, t_{r-1})} = \phi''_{uu}(t_{r-1}, B_m(\omega, \tilde{s}_r)).
\end{equation}
By the uniform continuity of $\phi''_{uu}$ on $[0,T] \times K(\omega)$, there exists $m_2(\omega)$ such that for each $m \ge m_1(\omega)$ we have
\begin{equation}\label{eq:locdrift3}
|\phi''_{uu}(t_{r-1},B_m(\omega, \tilde{s}_r)) - \phi''_{uu}(t_r,B_m(\omega, t_r))| < \frac{\epsilon}{2},
\end{equation}
since $|t_r - t_{r-1}| \le 2^{-2m}$ and $B_m(\omega, \tilde{s}_r) - B_m(\omega, t_r) \le 2^{-m}$.

By (\ref{eq:locdrift1}), (\ref{eq:locdrift2}), and (\ref{eq:locdrift3}), statement (\ref{eq:nu_m1}) of the lemma holds when $m \ge \max\{m_0(\omega), m_1(\omega), m_2(\omega)\}$. The first two equalities in (\ref{eq:nu_m}) follow from this and the uniform continuity of $\nu(t,x)$ on $[0,T]\times K(\omega)$, while the third equality follows from formula (\ref{eq:Skor}) and again from the uniform continuity of $\nu(t,x)$.
\end{proof}

For any $m = 0, 1, 2, \dots$ fixed, we now change the probability $\mathbb{P}$ into a probability measure $\mathbb{Q}_m$ to adjust the drift $\nu(t_{r}, B_m(t_{r}))$ in (\ref{eq:nu_m1}) to coincide with $\mu(t_{r}, X_m(t_{r}))$, where $X_m$ is defined by (\ref{eq:X_m}). Similarly to Girsanov's theorem, this can be achieved by introducing a new probability measure $\mathbb{Q}_m$ setting
\begin{multline}\label{eq:Girsm}
\frac{d\mathbb{Q}_m}{d\mathbb{P}} = \Lambda_m(T)  \\
:= \exp\left\{-\sum_{r=1}^{\lfloor T 2^{2m} \rfloor} \psi_m(t_{r-1}) (B_m(t_r) - B_m(t_{r-1})) \right. \\
- \left. \sum_{r=1}^{\lfloor T 2^{2m} \rfloor} \log \cosh \left(\psi_m(t_{r-1}) 2^{-m}\right) \right\},
\end{multline}
where
\begin{equation}\label{eq:psim}
\psi_m(t_r) := \frac{\nu(t_r, B_m(t_r)) - \mu(t_r, X_m(t_r))}{\sigma(t_r, X_m(t_r))} = \tilde{\psi}(t_r, B_m(t_r)) \qquad (t_r \in [0, T])
\end{equation}
and $\tilde{\psi}$ is defined by (\ref{eq:psitilde}).

Formula (\ref{eq:Girsm}) is based on the following lemma.
\begin{lem}
Under Assumptions 1 and 2, for any $m \ge 0$ fixed, taking $\Lambda_m(0) := 1$,
\begin{multline}\label{eq:lambdam}
\Lambda_m(t_n) := \exp\left\{-\sum_{r=1}^{n} \psi_m(t_{r-1}) (B_m(t_r) - B_m(t_{r-1})) \right. \\
\left. - \sum_{r=1}^{n} \log \cosh \left(\psi_m(t_{r-1}) 2^{-m}\right) \right\} \quad (n \ge 1)
\end{multline}
is a discrete time positive $\mathbb{P}$-martingale over $t_n \in [0, T]$, with
\[
\mathbb{E}_{\mathbb{P}}\Lambda_m(t_n) = 1 .
\]
For comparison with (\ref{eq:Q}), recall that
\begin{equation}\label{eq:logcosh}
\log \cosh(x) = \frac{x^2}{2} + O(x^4) \qquad (x \to 0).
\end{equation}
\end{lem}
\begin{proof}
Fix $m \ge 0$. Under our assumptions, $\tilde{\psi}$ in (\ref{eq:psitilde}) is a continuous function on $[0, T] \times \mathbb{R}$. Then for any $t_r \in [0, T]$ fixed, $\psi_m(t_r) = \tilde{\psi}(t_r, B_m(t_r))$ is a random variable that takes finitely many finite values. The same statement is true for $\Lambda_m(t_n)$ with $t_n \in [0, T]$ fixed; thus $\Lambda_m(t_n)$ is a bounded random variable and so  $\mathbb{E}_{\mathbb{P}}\Lambda_m(t_n) < \infty$.

Also, for $n \ge 0$,
\begin{multline*}
\mathbb{E}_{\mathbb{P}} \left(\Lambda_m(t_{n+1})|\mathcal{F}_{t_n}\right) \\
= \frac{\Lambda_m(t_n)}{\cosh(\psi_m(t_n) 2^{-m})} \, \mathbb{E}_{\mathbb{P}} \left(e^{-\psi_m(t_n) (B_m(t_{n+1})- B_m(t_n))} | \mathcal{F}_{t_n}\right)
= \Lambda_m(t_n) .
\end{multline*}
This proves the lemma.
\end{proof}

\begin{thm} \label{th:measureconv}
In addition to Assumptions 1, 2 and 3, suppose that $\mu, \sigma \in C^{1,1}([0, T] \times D)$ and $\phi''_{tt}, \phi'''_{uut}, \phi'''_{uuu}$ exist and are continuous on $[0, T]\times \mathbb{R}$. Then the total variation distance between the probability measures $\mathbb{Q}_m$ and $\mathbb{Q}$ tends to 0:
\[
\delta(\mathbb{Q}_m, \mathbb{Q}) := \sup_{A \in \mathcal{F}} |\mathbb{Q}_m(A) - \mathbb{Q}(A)| \to 0 \qquad (m \to \infty).
\]
\end{thm}
\begin{proof}
By Scheff\'e's theorem, see e.g. \cite[p. 224]{Bil1968}, it is enough to show that the Radon--Nikodym derivatives $d \mathbb{Q}_m/d \mathbb{P} = \Lambda_m(T)$ converge to $d \mathbb{Q}/d \mathbb{P} = \Lambda(T)$ $\mathbb{P}$-a.s.

Under our assumptions, $\tilde{\psi}(t,u) \in C^{1,1}([0, T]\times \mathbb{R})$. Remember that $\psi(t) = \tilde{\psi}(t, B(t))$ and $\psi_m(t) = \tilde{\psi}(t, B_m(t))$. Thus applying Theorem \ref{th:Ito} with $f(t,u) = \tilde{\psi}(t,u)$, it follows that
\[
\lim_{m \to \infty} \sum_{r=1}^{\lfloor T 2^{2m} \rfloor} \psi_m(t_{r-1}) (B_m(t_r) - B_m(t_{r-1})) = \int_0^T \psi(s) \, dB(s) \qquad \mathbb{P}\text{-a.s.}
\]

Let the probability 1 subset $\Omega_1 \subset \Omega$ and the compact set $K(\omega) \subset \mathbb{R}$ denote the same as in the proof of Lemma \ref{le:locdrift}, where $\omega \in \Omega_1$ is fixed. Then $\tilde{\psi}(t,u)$ is uniformly continuous on the compact set $[0, T] \times K(\omega)$, so by the a.s. uniform convergence of $B_m$ to $B$ it follows that
\begin{equation}\label{eq:psiconv}
\lim_{m \to \infty} \sup_{0 \le t \le T} \left| \psi_m(\omega,t) - \psi(\omega,t)\right| = 0 .
\end{equation}
Moreover, the functions $\psi_m(\omega,t)$ are uniformly bounded on $[0, T]$ for all $m$ large enough. Thus by (\ref{eq:logcosh}),
\begin{multline*}
\lim_{m \to \infty} \sum_{r=1}^{\lfloor T 2^{2m} \rfloor} \log \cosh \left(\psi_m(\omega, t_{r-1}) 2^{-m}\right)
= \frac12 \, \lim_{m \to \infty} \sum_{r=1}^{\lfloor T 2^{2m} \rfloor} \psi^2_m(\omega, t_{r-1}) 2^{-2m} \\
= \frac12 \, \int_0^T \psi^2(\omega,s) ds
\end{multline*}
This completes the proof of the theorem.
\end{proof}

Now we add a suitable drift to $B_m$, the resulting nearest neighbor random walk is
\begin{multline}\label{eq:Wm}
W_m(t_n) := B_m(t_n) + \sum_{r=1}^n \tanh\left(\psi_m(t_{r-1}) 2^{-m}\right) 2^{-m} \\
= B_m(t_n) + \sum_{r=1}^n \psi_m(t_{r-1}) 2^{-2m} + O(2^{-2m}),
\end{multline}
where we used that
\begin{equation}\label{eq:tanh}
\tanh(x) = x + O(x^3) \qquad (x \to 0),
\end{equation}
and that under our assumptions, $\psi_m$ is bounded as $m \to \infty$, but its bound may depend on $\omega$.

\begin{lem} \label{le:Wm_mart}
Under Assumptions 1 and 2, for any $m \ge 0$ fixed, $W_m(t_n)$ is a $\mathbb{Q}_m$-martingale over $t_n = n 2^{-2m} \in [0, T]$ $(n=0,1,2\dots)$.
\end{lem}
\begin{proof}
As is well-known, to show that $W_m(t)$ is a $\mathbb{Q}_m$-martingale it is enough to prove that $\Lambda_m(t_n) W_m(t_n)$ is a $\mathbb{P}$-martingale. It is clear that $\Lambda_m(t_n) W_m(t_n)$ takes finitely many finite values, so its expectation is finite. 

Using (\ref{eq:lambdam}) and (\ref{eq:Wm}), we obtain that
\begin{multline*}
\mathbb{E}_{\mathbb{P}}\left( \Lambda_m(t_{n+1}) W_m(t_{n+1}) \mid \mathcal{F}_{t_n} \right) \\
= \Lambda_m(t_{n}) \left\{ \frac{W_m(t_{n})}{\cosh\left(\psi_m(t_n) 2^{-m }\right)}
\mathbb{E}_{\mathbb{P}}\left( e^{-\psi_m(t_n) (B_m(t_{n+1}) - B_m(t_n))} \mid \mathcal{F}_{t_n} \right) \right. \\
+  \frac{1}{\cosh\left(\psi_m(t_n) 2^{-m }\right)} \\
\left. \times \left\{ \tanh\left(\psi_m(t_{n}) 2^{-m}\right) 2^{-m}
\mathbb{E}_{\mathbb{P}}\left( e^{-\psi_m(t_n) (B_m(t_{n+1}) - B_m(t_n))} \mid \mathcal{F}_{t_n} \right) \right. \right. \\
+ \left. \left. \mathbb{E}_{\mathbb{P}}\left( (B_m(t_{n+1}) - B_m(t_n)) e^{-\psi_m(t_n) (B_m(t_{n+1}) - B_m(t_n))} \mid \mathcal{F}_{t_n} \right) \right\} \right\} \\
= \Lambda_m(t_{n}) \left\{ W_m(t_{n}) + \frac{\sinh\left(\psi_m(t_n) 2^{-m}\right) 2^{-m}
-  \sinh\left(\psi_m(t_n) 2^{-m}\right) 2^{-m}}{\cosh\left(\psi_m(t_n) 2^{-m }\right)} \right\} \\
= \Lambda_m(t_{n}) W_m(t_{n}) .
\end{multline*}
\end{proof}

By (\ref{eq:Wm}), $W_m(t_n)$ $(0 \le t_n \le T)$ is a nearest neighbor random walk. On a time interval $[t_n, t_{n+1}]$ it can move up or down by the amount
\begin{eqnarray}\label{eq:Wmstep}
W_m(t_{n+1}) - W_m(t_n) &=& B_m(t_{n+1}) - B_m(t_n) + \tanh(\psi_m(t_n) 2^{-m}) \, 2^{-m} \nonumber \\
&=& B_m(t_{n+1}) - B_m(t_n) + \psi_m(t_n) 2^{-2m} + O\left(2^{-4m}\right) \nonumber \\
&=& \pm 2^{-m} + \psi_m(t_n) 2^{-2m} + O\left(2^{-4m}\right).
\end{eqnarray}
The steps of $W_m$ are neither independent, nor identically distributed in general, but $W_m$ is a $\mathbb{Q}_m$-martingale. Denote the conditional $\mathbb{Q}_m$ probability of its stepping up or down during the time interval $[t_n, t_{n+1}]$ given $\mathcal{F}_{t_n}$ by $q^+_m(t_n)$ and $q^-_m(t_n)$, respectively. Then it follows that 
\begin{multline*}
q^+_m(t_n) \left(2^{-m} + \tanh\left(\psi_m(t_n) 2^{-m}\right) 2^{-m}\right) \\
+ q^-_m(t_n) \left(2^{-m} - \tanh\left(\psi_m(t_n) 2^{-m}\right) 2^{-m}\right) = 0 .
\end{multline*}
For $n \ge 0$ it implies that
\begin{equation}\label{eq:qn}
q^{\pm}_m(t_n) = \frac12 \mp \frac12 \tanh\left(\psi_m(t_n) 2^{-m} \right)
= \frac12 \mp \frac12 \psi_m(t_n) 2^{-m}  + O\left(2^{-3m} \right).
\end{equation}

Based on this, now we can describe the discrete probability distributions $\mathbb{Q}_m$ $(m=0,1,2,\dots)$, that is, the probabilities of the paths of the processes $W_m(t_r)$, $B_m(t_r)$, and most importantly, of $X_m(t_r)$ as $t_r = r 2^{-2m} \in [0,T]$  $(r=0,1,2,\dots)$.
\begin{thm} \label{th:Qm}
Start with Assumptions 1 and 2. Take an arbitrary positive integer $n$ such that $n 2^{-2m} \in [0, T]$ and arbitrary numbers $\epsilon_r = \pm 1$ $(r=0,1,\dots,n)$. Then
\begin{multline*}
\mathbb{Q}_m\left(X_m(t_0) = x_0, X_m(t_r) = \phi(t_r, (\epsilon_1 + \cdots + \epsilon_r)2^{-m}), \, r=1,2,\dots,n \right) \\
= \prod_{r=0}^{n-1} q^{\epsilon_{r+1}}_m(t_{r}),
\end{multline*}
where
\[
q^{\epsilon_{r+1}}_m(t_{r}) := \frac12 - \frac12 \epsilon_{r+1} \tanh\left(\psi_m(t_{r}) 2^{-m} \right) .
\]
\end{thm}
\begin{proof}
By definition, $X_m(t_r) = \phi(t_r, B_m(t_r))$, so the paths of $X_m$ are determined by the paths of $B_m$, which are, in turn, determined by the paths of $W_m$ by (\ref{eq:Wm}) and (\ref{eq:Wmstep}). Since $W_m$ is a martingale by Lemma \ref{le:Wm_mart}, we can determine the probability of a path step-by-step, using the conditional $\mathbb{Q}_m$ probabilities (\ref{eq:qn}) given the past $\sigma$-algebra $\mathcal{F}_{t_{r}}$ at the step in the time interval $[t_{r}, t_{r+1}]$ and then multiplying these conditional probabilities.
\end{proof}

By piecewise linear interpolation, one can extend the process $W_m$ over the whole interval $[0,T]$ as processes with a.s. continuous paths. (Since the probability measures $\mathbb{P}$ and $\mathbb{Q}$ are equivalent, the notion ``almost sure'' refers to both.)
\begin{lem} \label{le:Wm_conv}
Under Assumptions 1, 2 and 3, the sequence of processes $W_m(t)$ a.s. converges as $m \to \infty$, uniformly on $[0,T]$, to the $\mathbb{Q}$-Brownian motion $W(t)$ defined by (\ref{eq:W}).
\end{lem}
\begin{proof}
As was shown in the proof of Theorem \ref{th:measureconv}), $\psi_m$ a.s. uniformly converges to $\psi$ and the functions $t \mapsto B_m(\omega, t)$ are uniformly bounded on $[0, T]$ for almost every $\omega$. Thus by (\ref{eq:tanh}) we get that
\begin{multline*}
\lim_{m \to \infty} W_m(t) = \lim_{m \to \infty} \left(B_m(t) + \sum_{r=1}^{\lfloor t 2^{2m} \rfloor} \tanh\left(\psi_m(t_{r-1}) 2^{-m}\right) 2^{-m}\right) \\
= \lim_{m \to \infty} \left(B_m(t) + \sum_{r=1}^{\lfloor t 2^{2m} \rfloor} \psi_m(t_{r-1}) 2^{-2m}\right)
= B(t) + \int_0^t \psi(u) du = W(t) ,
\end{multline*}
a.s., uniformly for $t \in [0,T]$.
\end{proof}

\begin{lem} \label{le:discSDEQ}
In addition to Assumptions 1, 2 and 3, let us suppose that $\sigma(t,x) \in C^{1,1}([0, T] \times \mathbb{R})$. Then the discrete diffusion
\begin{equation}\label{eq:solm}
X_m(t_n) := \phi(t_n, B_m(t_n)) = \phi\left(t_n, W_m(t_n) - \sum_{r=1}^n \psi_m(t_{r-1}) 2^{-2m} \right)
\end{equation}
$(t_n = n2^{-2m} \in [0,T])$ approximately satisfies a difference equation that corresponds to the SDE (\ref{eq:SDEQ}):
\begin{multline}\label{eq:discSDE}
X_m(t_n) - x_0
=  \sum_{r=1}^{n} \sigma\left(t_{r-1}, X_m(t_{r-1})\right) \: \left(W_m(t_r) - W_m(t_{r-1})\right)  \\
+  \sum_{r=1}^{n} \mu\left(t_{r-1}, X_m(t_{r-1})\right) 2^{-2m} + \text{error},
\end{multline}
where $|\text{error}| < \epsilon$ with arbitrary $\epsilon > 0$ when $m \ge m_0(\omega)$, uniformly for $t_n \in [0, T]$.
\end{lem}
\begin{proof}
Combine the discrete It\^o's formula (\ref{eq:disc_Ito2}) with (\ref{eq:appr1}), (\ref{eq:Wmstep}), and Lemma \ref{le:locdrift}, using the identity $\phi'_u\left(t_{r-1}, B_m(t_{r-1})\right) = \sigma\left(t_{r-1}, X_m(t_{r-1})\right)$ as well:
\begin{multline*}
X_m(t_n) - x_0 = X^*_m(t_n) - x_0 + O(2^{-m}) \\
=  \sum_{r=1}^{n} \sigma\left(t_{r-1}, X_m(t_{r-1})\right) \: \left(W_m(t_r) - W_m(t_{r-1})\right)  \\
- \sum_{r=1}^{n} \sigma\left(t_{r-1}, X_m(t_{r-1})\right) \psi_m(t_{r-1}) 2^{-2m}
+  \sum_{r=1}^{n} \nu(t_r, B_m(t_r)) 2^{-2m} + \text{error},
\end{multline*}
where $|\text{error}| < \epsilon$ with arbitrary $\epsilon > 0$ when $m \ge m_0(\omega)$, uniformly for $t_n \in [0, T]$.

Then use definitions (\ref{eq:psim}) for $\psi_m(t_{r-1})$ to obtain
\begin{multline*}
X_m(t_n) - x_0
=  \sum_{r=1}^{n} \sigma\left(t_{r-1}, X_m(t_{r-1})\right) \: \left(W_m(t_r) - W_m(t_{r-1})\right)  \\
+ \sum_{r=1}^{n} \mu\left(t_{r-1}, X_m(t_{r-1})\right) 2^{-2m} \\
-  \sum_{r=1}^{n} \nu(t_{r-1}, B_m(t_{r-1})) 2^{-2m} +  \sum_{r=1}^{n} \nu(t_r, B_m(t_r)) 2^{-2m} + \text{error}.
\end{multline*}

Cancelling the corresponding terms in the last two sums, only two terms remain:
\[
\left\{\nu(t_n, B_m(t_n)) - \nu(0, 0)\right\} 2^{-2m} = O(2^{-2m}).
\]
Here we used that by (\ref{eq:Skor}) and the continuity of $\nu(t,u)$, for almost every $\omega$, $|\nu(t_n, B_m(t_n))|$ is uniformly bounded when $t_n \in [0, T]$ and $m > m_0(\omega)$. This completes the proof of the lemma.
\end{proof}


\section{An algorithm and examples} \label{sec:AlgEx}


\subsection{An algorithm to approximate a diffusion} \label{ssec:Alg}

Under certain assumptions on the coefficients (see Section \ref{sec:Disc}), we have developed a method of approximation of an It\^o diffusion. First we gave a weak solution $X(t)=\phi(t, B(t))$, $t \in [0, T]$, of the SDE (\ref{eq:SDE}) w.r.t. a new probability measure $\mathbb{Q}$. Then for any $m \ge 0$ we took $X_m(t_r) = \phi(t_r, B_m(t_r))$, $t_r=r 2^{-2m} \in [0, T]$ which has steps with suitable $\mathbb{Q}_m$ probabilities approximating the probability measure $Q$. Theorem \ref{th:Xconv} showed that $X_m$ a.s. uniformly converges to $X$ on $[0,T]$. Here we describe the steps of an algorithm realizing this method.

First solve the ODE $\phi'_u(t, u) = \sigma(t, \phi(t,u))$ with the initial value $\phi(t,0)=x_0$. By  Assumption 2 in Section \ref{sec:Disc}, $\phi$ is a $C^{1,2}$ solution over $[0, T] \times \mathbb{R}$. An analytic solution may be substituted by a numerical one, even step-by-step as the algorithm goes as follows. The algorithm needs $\phi'_t$ and $\phi''_{uu}$ or their approximations as well.

Second, choose a suitable value of $m$: usually, $m=5$ or $m=6$ will do. Start with the initial values $B_m(0) = 0$ and $X_m(0) = x_0$. Proceed with time steps $t_r = r 2^{-2m}$, $(r=0,1,2,\dots)$ consecutively and determine the $\mathbb{Q}_m$ probability of the next step of the shrunken random walk $B_m$ on time interval $[t_{r}, t_{r+1}]$:
\[
\nu(t, u) = \phi'_t(t, u)) + \frac12 \phi''_{uu}(t, u),
\]
\[
\psi_m(t_{r}) = \frac{\nu(t_{r}, B_m(t_{r})) - \mu(t_{r}, X_m(t_{r}))}{\sigma(t_{r}, X_m(t_{r}))} .
\]
By Theorem \ref{th:Qm}, the conditional $\mathbb{Q}_m$ probability of an up-step of $B_m$ given $\mathcal{F}_{t_{r}}$ is
\[
q^{+}_m(t_{r}) = \frac12 - \frac12 \tanh\left(\psi_m(t_{r}) 2^{-m} \right) ,
\]
Generate a random number $U_r$, uniformly distributed on $[0, 1]$. Let
\[
B_m(t_{r+1}) = \left\{\begin{array}{ll}
                        B_m(t_r) + 2^{-m} & \text{if} \quad U_r \le  q^{+}_m(t_{r}), \\
                        B_m(t_r) - 2^{-m} & \text{otherwise.}
                      \end{array}
               \right.
\]
Take $X_m(t_{r+1}) = \phi(t_{r+1}, B_m(t_{r+1}))$.


\subsection{Some examples} \label{ssec:Ex}

In the following examples we consider some simple, time-homogeneous diffusions; that is, their drift $\mu(x)$ and diffusion coefficient $\sigma(x)$ do not depend explicitly on time.

\begin{enumerate}[(1)]

\item Let $\sigma(x) = ax+b$, where $a > 0$ and $b \ge 0$. The corresponding ODE is
\[
\phi' = a \phi + b, \qquad \phi(0) = x_0 > 0.
\]
Its solution is
\[
x = \phi(u) = \left(x_0 + \frac{b}{a}\right) e^{au} - \frac{b}{a}, \qquad u \in \mathbb{R}.
\]
Then our standing assumptions in Section \ref{sec:Meth} and the conditions of Lemma \ref{le:suffice} hold. The condition for the drift in Lemma \ref{le:suffice} can be rewritten here as
\begin{equation}\label{eq:linear}
|\mu(x)| \le \widetilde{K} (ax+b) \left(1 + \left| \log (ax+b)\right|\right), \qquad  -b/a < x < \infty,
\end{equation}
with some $\widetilde{K} < \infty$.

The case of \emph{geometric Brownian motion}: $\mu(x) = c x$ and $\sigma(x) = a x$, where $a > 0$ and $c \in \mathbb{R}$. Then
\[
\phi(u) = x_0 e^{au}, \quad \nu(u) = \frac{a^2 x_0}{2} e^{au}, \quad \psi_m(t_r) = \frac{a}{2} - \frac{c}{a},
\]
and (\ref{eq:linear}) clearly holds, our method is applicable. Figures \ref{fig:geom_BM_m=5} and \ref{fig:geom_BM_m=6} show typical sample paths of the approximate and the exact solutions in this case, with $m=5$ and $m=6$, respectively, while $T=5$, $a=c=x_0=1$.
\begin{figure}[ht]
   \begin{center}
     \includegraphics[scale=0.65]{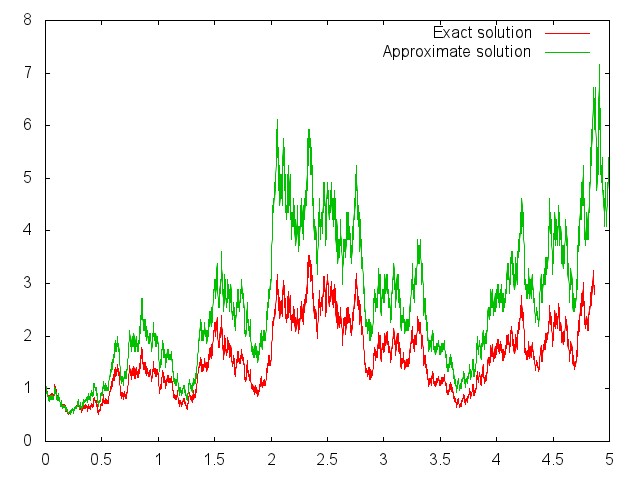}
      \caption{Approximation of geometric Brownian motion, m=5}
      \label{fig:geom_BM_m=5}
    \end{center}
\end{figure}

\bigskip

\begin{figure}[ht]
   \begin{center}
     \includegraphics[scale=0.65]{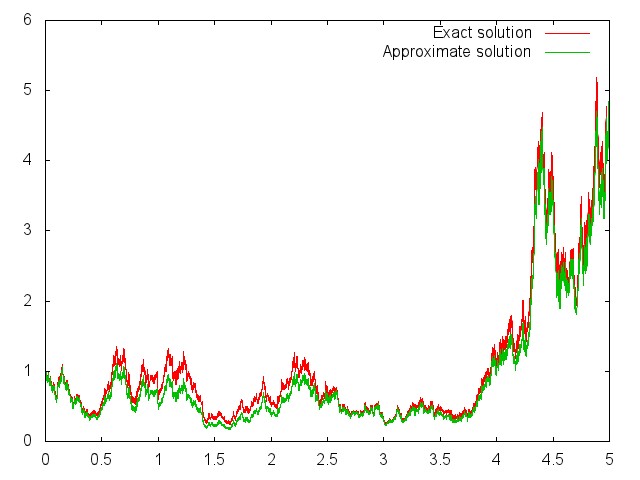}
      \caption{Approximation of geometric Brownian motion, m=6}
      \label{fig:geom_BM_m=6}
    \end{center}
\end{figure}

Unfortunately, it is easy to find examples where our method \emph{does not work}. Let $\mu(x) = - d $ and $\sigma(x) = a x$, where $a,d, x_0  > 0$. Then
\[
\nu(u) = \frac{a^2 x_0}{2} e^{au}, \quad \psi_m(t_r) = \frac{a}{2} + \frac{d}{a x_0} e^{-a B_m(t_r)},
\]
and (\ref{eq:linear}) does not hold. In fact, $\phi(B(t)) = x_0 e^{a B(t)}$ supplied by our method cannot be a solution of this SDE, because $\phi(B(t))$ is positive for all $t$, but the true solution $X(t)$ is negative with positive probability, at least when $t$ is large enough. The latter statement follows from the fact that the unique strong solution of this SDE is
\begin{multline*}
      X(t) = e^{a B(t) - a^2t/2} \left\{x_0 - d \int_0^t e^{-a B(s) + a^2s/2} d s \right\} \\
      =: Y(t) \left\{x_0 - d \int_0^t (Y(s))^{-1} d s\right\},
\end{multline*}
where $Y(t)$ has a lognormal distribution with expectation $e^{a^2t}$. Thus $Z(t) := \int_0^t (Y(s))^{-1} d s$ is a positive random variable with expectation $(e^{a^2t} - 1)/a^2$ and $\mathbb{P}\left( Z(t) > x_0/d \right) > 0$, at least when $t$ is large enough.
This fact also shows that the process $\Lambda(t)$ given by (\ref{eq:Lambda}) is not a $\mathbb{P}$-martingale in this example. Figure \ref{fig:counter} shows typical sample paths of the approximate and the exact solutions in this case, with $m=5$, $T=5$, $a=d=x_0=1$.
\begin{figure}[ht]
   \begin{center}
     \includegraphics[scale=0.65]{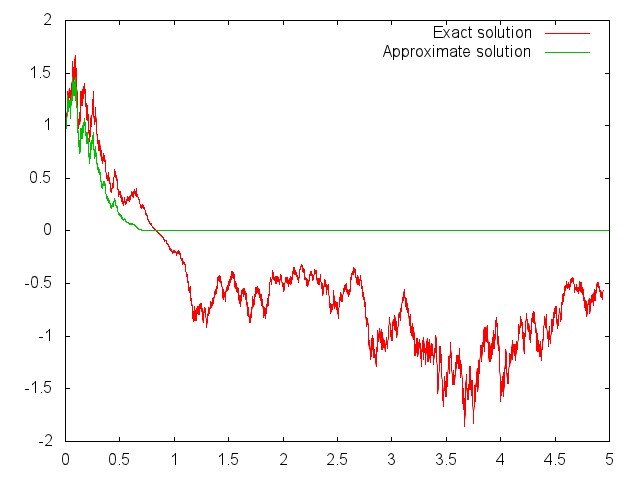}
      \caption{A counterexample, m=5}
      \label{fig:counter}
    \end{center}
\end{figure}

\item Let $\sigma(x) = b > 0$. Then our ODE is $\phi' = b$, $\phi(0) = x_0 \in \mathbb{R}$. Its solution is $x = \phi(u) = bu + x_0$. Thus the condition for the drift in Lemma \ref{le:suffice} can be rewritten here as
\begin{equation}\label{eq:driftineq}
|\mu(x)| \le \widetilde{K} \left(1 + |x|\right), \qquad x \in \mathbb{R},
\end{equation}
with some $\widetilde{K} < \infty$.

In particular, in the case of an \emph{Orstein--Uhlenbeck process}, where $\mu(x) = cx+d$ ($c, d \in\mathbb{R}$) and $\sigma(x) = b > 0$,
\[
\nu(u) = 0, \quad \psi_m(t_r) = - \frac{c x_0}{b} - c B_m(t_r).
\]
Here (\ref{eq:driftineq}) does hold and our method is applicable.

\end{enumerate}

\subsection*{Acknowledgement}
The authors are indebted to Tibor Homoki (currently: MSc student at BUTE) for writing a Python simulation program of the algorithm described above and producing many useful results, including the attached figures.


\end{document}